\documentclass{amsart}
\usepackage{amsmath}
\usepackage{amsfonts}
\usepackage{amssymb}
\usepackage{graphicx}%
\usepackage{mathrsfs}

\usepackage{amsthm}         
\usepackage{hyperref} 
\usepackage{url}
\usepackage{multicol}

\usepackage{tikz,comment}
\usetikzlibrary{arrows,calc,intersections,quotes,angles}
\tikzset{>=stealth'}

\usepackage{times}
\usepackage{bbold}
\usepackage{stmaryrd}
\usepackage{array}
\usepackage{dsfont}
\usepackage[T1]{fontenc}

\newtheorem{theorem}{Theorem}

\newtheorem*{GGR+theorem}{GGR+ Theorem}

\newtheorem{lemma}[theorem]{Lemma}

\newtheorem*{PWC}{The Perimeter Winternitz Conjecture}

\begin{document}
\title[The Perimeter Winternitz Theorem in a Triangle]{The Perimeter Winternitz Theorem in a Triangle}

\author{Allan Berele}
\address{Department of Mathematics, DePaul University, Chicago, IL 60614}
\email{aberele@depaul.edu}

\author{Stefan Catoiu}
\address{Department of Mathematics, DePaul University, Chicago, IL 60614}
\email{scatoiu@depaul.edu}

\date{March 24, 2026}
\subjclass[2010]{Primary: 52A10; 52A38; Secondary: 51M04; 51M25; 51N20.}
\keywords{Winternitz theorem.}

\begin{abstract} A variable line through the centroid $G$ of a triangle divides the triangle into~two parts each of whose lengths as a fraction of the perimeter fills a closed interval $[m,1-m]$, with~$m$ between~0 and~1/2. We show that the range of $m$ taken over all triangles is the interval $(3/10,4/9]$, with 3/10 approached by scales of the triangles approaching the~5-4-1 triangle and their mid-size medians, and 4/9 attained by the equilateral triangles and the lines through~$G$ parallel to the sides.
This result is the perimeter version of the classical Winternitz theorem for a triangle, asserting that, in the case of area-ratio instead of perimeter-ratio, $m=4/9$, and this is attained by all triangles and their lines through $G$ and parallel to~the~sides..\vspace{-.1in}
\end{abstract}
\maketitle

\noindent
The Winternitz theorem asserts that each of the two regions determined by a variable line through the centroid $G$ of a planar convex set covers at least~4/9 of the total area, and that the bound 4/9 is precisely attained when the convex set is a triangle and the line through $G$ is parallel to a side. See Blaschke \cite{B}.

In their book \cite{YB}, first published in 1951, Yaglom and Boltyanskii suggested the perimeter version of the Winternitz theorem as a good project to work on. However, perhaps due to its being a very hard question, nobody has ever attempted to attack this problem since then.

The goal of this paper is to prove the perimeter-analogue of the Winternitz theorem for a triangle, and we shall see that this is an interesting problem as well.
\begin{figure}[b]
\begin{center}\vspace{-.15in}
\begin{tikzpicture}[scale=2.3]\small
\path [fill=gray!30] (.39,0) -- (2,0) -- (1.42,1.00) -- cycle;
\path [draw,thick] (0,0) -- (2,0) -- (1,1.73) -- cycle;
\node [above] at (1,1.73) {$A$};
\node [below] at (0,0) {$B$};
\node [below] at (2,0) {$C$};
\draw [fill=black] (1,.5766) circle [radius=.02];\node [left] at (1,.5766) {$G$};
\draw (.333,.5766) -- (1.667,.5766);
\draw (.66,0) -- (1.33,1.1533);
\draw (1.33,0) -- (.66,1.1533);
\draw (.333,.5766) -- (.667,0);
\draw (1.333,0) -- (1.667,.5766);
\draw (.66,1.1533) -- (1.33,1.1533);
\draw [dashed] (.25,-.135) -- (1.55,1.115);
\node [below] at (.39,0) {$M$};\node [right] at (1.42,1.00) {$N$};
\node [below] at (.667,0) {$M_1$};\node [above right] at (1.333,1.1533) {$N_1$};
\draw [fill=black] (.57,.17) circle [radius=.02];\node [above] at (.58,.17) {$P$}; 
\end{tikzpicture}
\qquad
\begin{tikzpicture}[scale=2.3]\small
\draw [color=gray,line width=3pt] (.39,0) -- (2,0) -- (1.42,1.00);
\path [draw,thick] (0,0) -- (2,0) -- (1,1.73) -- cycle;
\node [above] at (1,1.73) {$A$};
\node [below] at (0,0) {$B$};
\node [below] at (2,0) {$C$};
\draw [fill=black] (1,.5766) circle [radius=.02];\node [left] at (1,.5766) {$G$};
\draw (.333,.5766) -- (1.667,.5766);
\draw (.66,0) -- (1.33,1.1533);
\draw (1.33,0) -- (.66,1.1533);
\draw (.333,.5766) -- (.667,0);
\draw (1.333,0) -- (1.667,.5766);
\draw (.66,1.1533) -- (1.33,1.1533);
\draw [dashed] (.25,-.135) -- (1.55,1.115);
\node [below] at (.39,0) {$M$};\node [right] at (1.42,1.00) {$N$};
\node [below] at (.667,0) {$M_1$};\node [above right] at (1.333,1.1533) {$N_1$};
\draw [fill=black] (.57,.17) circle [radius=.02];\node [above] at (.58,.17) {$P$}; 
\end{tikzpicture}\vspace{-.1in}
\end{center}
\caption{Proof of the Winternitz theorem for the equilateral triangle and its perimeter analogue.}\label{fig1.1}\vspace{-.15in}
\end{figure}
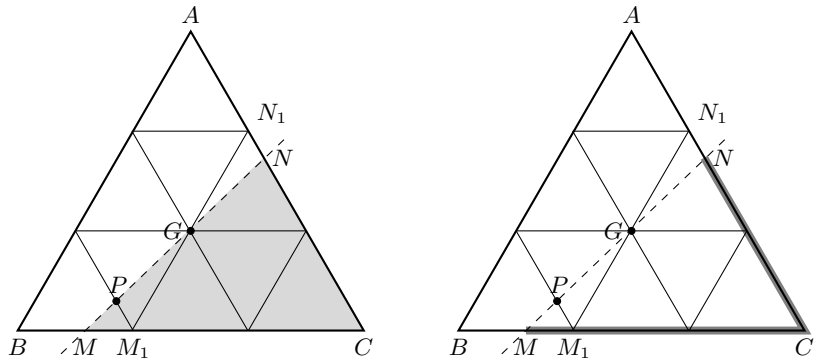
\subsection*{The case of an equilateral triangle.} The classical proof of the Winternitz theorem has two parts, the proof of the theorem for all triangles and the reduction from convex sets to triangles. The Winternitz theorem for triangles has the following simplified form: in any triangle, the smallest part of the area cut by a variable line through the centroid is 4/9 of the total area, and this is attained precisely when the line is parallel to a side. Affine invariance makes the proof of the Winternitz theorem for all triangles equivalent to the proof for an equilateral triangle, illustrated in~Figure~\ref{fig1.1}~(left), where the triangle is divided into nine smaller congruent triangles by parallel lines to the sides. Suppose the equilateral triangle $ABC$ has side $a$ and area~$\sigma =[ABC]$. Based on the symmetries of the equilateral triangle and following notation from the figure, without loss of generality, we may assume that the variable line $MN$ through the centroid~$G$ crosses the perimeter as shown in the figure. To prove the theorem, it suffices to prove that~$\tfrac 49\sigma\leq [CMN]\leq \tfrac 12\sigma $, or $[CM_1N_1]\leq [CMN]\leq [CBB']$, where $B'$ (not shown) is the midpoint of $AC$. We only prove the first of these two inequalities; the second has a similar proof. Indeed,
\[
\begin{aligned}
{[CMN]-[CM_1N_1]}&=\left([GMM_1]+[GNCM_1]\right)-\left([GNN_1]+[GNCM_1]\right)\\
&=[GMM_1]-[GNN_1]\\
&=\left([PMM_1]+[GPM_1]\right)-[GNN_1]\\
&=[PMM_1]\geq 0,
\end{aligned}
\]
where the last equality follows from the symmetry of triangles $GPM_1$ and $GNN_1$ relative to $G$. The inequality is equality precisely when $M=M_1$ and $N=N_1$.

The perimeter analogue of the Winternitz theorem for an equilateral triangle asserts that the length of the shortest perimeter part cut from such a triangle by a variable line through the centroid is 4/9 of the perimeter, and 4/9 is attained precisely when the line is parallel to a side. Its analogous proof is illustrated in Figure~\ref{fig1.1}~(right), where it suffices to show that $\tfrac 49\cdot (3a)\leq CM+CN\leq \tfrac 12\cdot (3a)$, that is,
$CM_1+CN_1\leq CM+CN\leq CB+CB'$. We only prove the first of the two similar inequalities. For this, we compute
\[
\begin{aligned}
{(CM\hspace{-.03in}+\hspace{-.03in}CN)-(CM_1\hspace{-.03in}+\hspace{-.03in}CN_1)}&=(MM_1\hspace{-.03in}+\hspace{-.03in}CM_1\hspace{-.03in}+\hspace{-.03in}CN)-(CM_1\hspace{-.03in}+\hspace{-.03in}CN\hspace{-.03in}+\hspace{-.03in}NN_1)\\
&=MM_1\hspace{-.03in}-\hspace{-.03in}NN_1=MM_1\hspace{-.03in}-\hspace{-.03in}PM_1\geq 0,
\end{aligned}
\]
where the inequality is equivalent in triangle $PMM_1$ to $\angle P-\angle M\geq 0$, and this follows from $\angle M\leq 60^{\circ }$ and $\angle M_1=60^{\circ }$. The same inequality becomes equality when $\angle M=60^{\circ }$, i.e., when $M=M_1$ and $N=N_1$.

In absence of the affine transformation in the perimeter case, the above perimeter-analogue of the Winternitz theorem for the equilateral triangle does not easily extend to a perimeter-analogue of the Winternitz theorem for all triangles. Therefore, new ideas are required. In particular, the quest to proving such an extension fills up the goal of the entire article.\vspace{-.15in}
\[
\ast\quad\ast\quad\ast\vspace{-.05in}
\]%
Our main result is the following extension to all triangles of the above perimeter-analogue of the Winternitz theorem for the equilateral triangle.

\begin{theorem}[The Perimeter Winternitz Theorem in Every Triangle]\label{T1} In every triangle, the shortest perimeter part cut by
a variable line through the centroid covers between~3/10 and~4/9 of the perimeter. 

The bound~3/10 is approached by scales of triangles approaching the~5-4-1 degenerated triangle and their medians corresponding to the middle sides. The bound~4/9 is attained by the equilateral triangles and their lines through the centroid and parallel to the sides. See~Figure~\ref{fig1}.
\end{theorem}
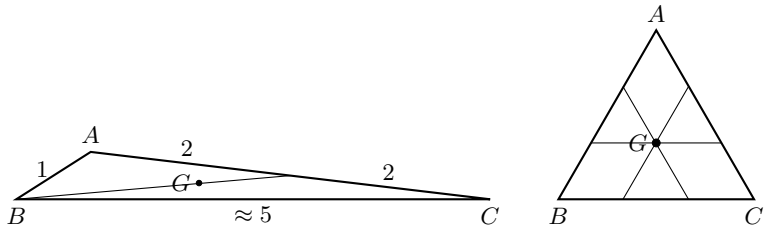
\begin{figure}[t]
\begin{center}\vspace{-.15in}
\begin{tikzpicture}[scale=1.8]\small
\path [draw,thick] (0,0) -- (3.5,0) -- (.55,.35) -- cycle;
\node [above] at (.55,.35) {$A$};
\node [below] at (0,0) {$B$};
\node [below] at (3.5,0) {$C$};
\draw [fill=black] (1.35,.12) circle [radius=.02];\node [left] at (1.35,.12) {$G$};
\node [left] at (.3,.22) {$1$};\node [above] at (1.26,.26) {$2$};
\node [above] at (2.75,.08) {$2$};\node [below] at (1.75,.01) {$\approx 5$};
\draw (0,0) -- (2.025,.175);
\end{tikzpicture}
\quad
\begin{tikzpicture}[scale=1.3]\small
\path [draw,thick] (0,0) -- (2,0) -- (1,1.73) -- cycle;
\node [above] at (1,1.73) {$A$};
\node [below] at (0,0) {$B$};
\node [below] at (2,0) {$C$};
\draw [fill=black] (1,.5766) circle [radius=.04];\node [left] at (1,.5766) {$G$};
\draw (.333,.5766) -- (1.667,.5766);
\draw (.66,0) -- (1.33,1.1533);
\draw (1.33,0) -- (.66,1.1533);
\end{tikzpicture}\vspace{-.1in}
\end{center}
\caption{The extremal triangles in Theorem~\ref{T1} together with all their chords through the centroid~$G$ that respectively separate approximately 3/10 and exactly 4/9 of the perimeter.}\label{fig1}\vspace{-.15in}
\end{figure}

Note that unlike the area case of the Witernitz theorem, where $m=4/9$ for all triangles, in the perimeter case of Theorem~\ref{T1}, $m$ ranges over an entire interval whose one endpoint is $4/9$. In particular, $m$ depends on the shape of the triangle and this explains the difficulty in extending the above proof of the perimeter Winternitz theorem for an equilateral triangle to a proof for all triangles.

The proof of Theorem~\ref{T1} is given in the last section and this is not quite the extension to all triangles of its above proof for an equilateral triangle. The proof has two parts. First, in~Theorem~\ref{T2} we compute the expression for the minimum part of the perimeter cut by a variable line through the centroid of a triangle in terms of the side lengths. Then, use the multivariable calculus method to determine the range of this expression as a fraction of the perimeter.

Let $MN$ be a variable transversal through the centroid $G$ of a triangle $ABC$ with sides $a=BC$, $b=AC$, $c=AB$ and semiperimeter $s$. If~$\theta $ is the inclination angle of the line $MN$, then the area of each of the two regions determined by~$MN$ on~$ABC$, as a fraction of the total area, is a continuous function of $\theta $, hence its range is a closed interval $[m,1-m]$, where~$m$ is between 0 and 1/2. The classical Winternitz theorem asserts that $m=4/9$.
In this article, $w$ is the minimum length of either of the two pieces cut from the perimeter by the same variable line, and $m=w/(2s)$ is the fraction of the perimeter represented by $w$.
\[
\ast\quad\ast\quad\ast
\]
In his article \cite{N}, published in 1945, Neumann proved two theorems. The first is a reproof of the Winternitz theorem, apparently unaware of Winternitz.  The second is the following generalization of the Winternitz theorem for continuous mass distributions: for every compact convex set in the plane and every continuous mass distribution, there is a point $P$ with the property that every line through $P$ divides the set into two parts, each of which covering at least 1/3 of the total mass of the set. In addition, the bound of~1/3 is optimal. Both the Neumann and Winternitz theorems were extended to~$n$ dimensions by Gr\"unbaum in \cite{G} in 1960, with 1/3 replaced by~$1/(n+1)$ and~4/9 replaced by $(n/(n+1))^n$. 
Neumann noted that, in the case when the mass is equally distributed along the perimeter, one can do better than 1/3 for planar convex sets, and conjectured without support that the best bound in this case would be~$(3-\sqrt{5})/2\approx .382$.~To the best of our knowledge, Neumann's conjecture is still open today, while its extension to $n$ dimensions, the Gr\"unbaum problem (which does not have a conjectured value yet), is one of the most outstanding open problems in convex geometry. After Gr\"unbaum, most of the subsequent work on the subject has been at the upper (Winternitz) end of his two results, an area referred to generically as the Gr\"unbaum inequality. For example, the Winternitz theorem was recently generalized by Shintar and Yaskin in \cite{SY}, and by Bose, Carmi, Hurtado, and Morin in~\cite{BCHM}.

Back to the main result here, our constant of 3/10 being smaller than 1/3 shows that the centroid of a convex set in the plane cannot play the role of the point $P$ in Neumann's theorem and conjecture. This is probably the reason why our natural perimeter analogue of the Winternitz theorem for triangles remained undetected during the century old history of the subject.
\section{Perimeter Winternitz lines.}\label{S1}
Given an angle with vertex at $A$ and an interior point $P$, a line through $P$ that intersects the sides of the angle at $M$ and $N$ is called a \emph{W-line} (relative to angle or vertex $A$) if it minimizes $AM+AN$. Note that the value $AM+AN$ decreases from infinity, when $MN$ is parallel to one of the sides, to the minimum attained by the W-line, and then increases back to infinity, when~$MN$ is parallel to the other side of the angle.

The following lemma on W-lines will be crucial in the proof of all remaining results.

\begin{lemma}\label{L8}
\begin{enumerate}
\item[(i)\,] Given a triangle $ABC$ and a point $P$ inside it, a line through $P$ maximizing/minimizing the perimeter of one of the two pieces cut by it is either a cevian or a W-line with respect to one of the angles.
\item[(ii)] Given an angle $\angle A$, let $\vec{u}$ and $\vec{v}$ be unit vectors along the sides, and let $P=x\vec{u}+y\vec{v}$ be an interior point. Then the W-line through $P$ intersects the sides of the angle at $(x+\sqrt{xy})\vec{u}$ and $(y+\sqrt{xy})\vec{v}$. In particular, the perimeter cut off from the angle by the W-line is  $x+y+2\sqrt{xy}=(\sqrt{x}+\sqrt{y})^2$.
\end{enumerate}
\end{lemma}

\begin{proof}
(i) This part is clear from the definition of a W-line. (ii) A line through $P$ intersecting the sides of the angle at $s\vec{u}$ and $t\vec{v}$ has equation $x/s+y/t=1$. In order to find the extreme values of $s+t$ we use Lagrange multipliers, and set\vspace{-.05in}
\[
F(s,t,\lambda )=s+t+\lambda \left(\frac xs+\frac yt-1\right).\vspace{-.05in}
\]
Then $F_s=1-{\lambda x}/{s^2}$ and $F_t=1-{\lambda y}/{t^2}$. Since these are zero, ${s^2}/{t^2}=x/y$. Combining with $x/s+y/t=1$ gives $s=x+\sqrt{xy}$ and $t=y+\sqrt{xy}$.
\end{proof}

\begin{figure}
\begin{center}
\begin{tikzpicture}[scale=1]\small
\path [draw,thick] (0,0) -- (5,0) -- (1.8,2.4) -- cycle;
\node [above] at (1.8,2.4) {$A$};
\node [below] at (0,0) {$B$};
\node [below] at (5,0) {$C$};
\draw [fill=black] (2.266,.8) circle [radius=.04];\node [above] at (2.26,.82) {$P$};
\draw (.5072,.676) -- (3.79,.907);
\draw (1.374,1.832) -- (2.96,0);
\draw (1.843,0) -- (2.74,1.6944);
\node [left] at (.5072,.676) {$M$};\node [right] at (3.79,.907) {$N$};
\end{tikzpicture}
\quad
\begin{tikzpicture}[scale=1]\small
\path [draw,thick] (0,0) -- (5,0) -- (1.8,2.4) -- cycle;
\node [above] at (1.8,2.4) {$A$};
\node [below] at (0,0) {$B$};
\node [below] at (5,0) {$C$};
\draw [fill=black] (2.266,.8) circle [radius=.04];\node [above] at (2.266,.8) {$P$};
\draw [dashed] (.6,.8) -- (3.933,.8);
\draw [dashed] (1.66,0) -- (2.867,1.6);
\draw [dashed] (3.33,0) -- (1.2,1.6);
\node [above right] at (2.3,1.95) {$y_A$};
\node [above right] at (4.42,.36) {$x_C$};
\node [above left] at (1.55,1.95) {$x_A$};
\node [above left] at (.36,.32) {$y_B$};
\node [below] at (.833,0) {$x_B$};
\node [below] at (4.166,0) {$y_C$};
\end{tikzpicture}\vspace{-.1in}
\end{center}
\caption{(Left) A triangle with the three perimeter W-lines relative to vertices passing through an interior point $P$. (Right) The same triangle with the parallels from $P$ to the sides.}\label{fig2}\vspace{-.1in}
\end{figure}
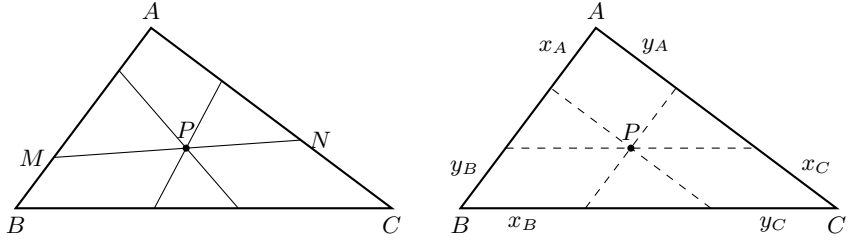

Given an interior point $P$ of a triangle $ABC$, we let $w_A$ be the minimum sum $AM+AN$ for a variable chord $MN$ through $P$, with $M$ on $AB$ and $N$ on $AC$. By Lemma~\ref{L8}, this is achieved when the chord $MN$ is either a W-line through $P$ relative to $A$, as shown in~Figure~\ref{fig2}~(left), or a cevian. Similar definitions are given for $w_B$ and~$w_C$. Let $w=w(P)$ be the length of the smallest part of the perimeter determined by a variable line through $P$.

The first application of Lemma~\ref{L8} relates the numbers $w$, $w_A$, $w_B$, $w_C$ associated to  the point $P$.

\begin{lemma}\label{L3} $w=\min \{w_A,w_B,w_C\}$.
\end{lemma}

\begin{proof} That $w\leq \min \{w_A,w_B,w_C\}$ follows from
the defining set for $w$ containing the union of the defining sets for $w_A$, $w_B$, and $w_C$. Without loss of generality, suppose $w$ is the shortest part of the perimeter determined by the chord $M_0N_0$ through $P$, with $M_0$ on $AB$ and $N_0$ on $AC$.
There are two cases to consider. If $w=AM_0+AN_0$, then $w\geq w_A\geq \min \{w_A,w_B,w_C\}\geq w$, and the equality of the first and last term forces the result. Otherwise, $w=M_0B+BC+CN_0$, hence $w= \min \{MB+BC+CN\mid M\in AB,N\in AC, P\in MN\}=2s-\max \{AM+AN\mid M\in AB,N\in AC, P\in MN\}$. Since the maximum (hence the previous minimum) cannot be attained by a W-line, by Lemma~\ref{L8}(i), they are attained when~$M_0N_0$ is a cevian from $B$ or $C$. Without loss of generality, assume that $M_0N_0$ is the cevian $CC'$. Then $w=C'B+BC$ belongs to the defining set of $w_B$, hence~$w\geq w_B\geq \min \{w_A,w_B,w_C\}\geq w$ and, again, the result is forced by the equality of the first and last terms.
\end{proof}

\section{Proof of Theorem~\ref{T1}.}\label{S2}
The parallels from the centroid to the sides of a triangle divide the sides into thirds, so that with the notation from Figure~\ref{fig2}~(right), when $P=G$,~$x_B=y_C=a/3$, $x_C=y_A=b/3$, and $x_A=y_B=c/3$. Then, by~Lemma~\ref{L8}(ii), the W-lines relative to the angles at $A,B,C$ respectively separate the finite lengths $b/3+c/3+2\sqrt{bc}/3$, $a/3+c/3+2\sqrt{ac}/3$, $a/3+b/3+2\sqrt{ab}/3$ from the infinite perimeters of the extended angles of the triangle.

The second application of Lemma~\ref{L8} is the following perimeter Winternitz theorem in a specific triangle. 

\begin{theorem}[The Perimeter Winternitz Theorem in a Specific Triangle]\label{T2} With the above notation, every triangle $ABC$ satisfies the following properties:
\begin{enumerate}
\item[(i)\;] If $a\geq b\geq c$ then\vspace{-.1in}
\[
\begin{aligned}
w_A&=\begin{cases}
b/3+c/3+2\sqrt{bc}/3,&\text{for }b\leq 4c,\\
c+b/2,&\text{for }b>4c,
\end{cases}\\
w_B&=\begin{cases}
a/3+c/3+2\sqrt{ac}/3,&\text{for }a\leq 4c,\\
c+a/2,&\text{for }a>4c,\vspace{-.05in}
\end{cases}\\
w_C&=a/3+b/3+2\sqrt{ab}/3.
\end{aligned}
\]
\item[(ii)\,] If $a\geq b\geq c$, then $w_A\leq w_B\leq w_C$. In particular, $w=w_A$.
\item[(iii)] $w\leq 4(a+b+c)/9$, with equality if and only if $a=b=c$.
\end{enumerate}
\end{theorem}

\begin{proof}
(i) By Lemma~\ref{L8}, $w_A=b/3+c/3+2\sqrt{bc}/3$ when~$x_A+\sqrt{x_Ay_A}=c/3+\sqrt{bc}/3\leq c$ and~$y_A+\sqrt{x_Ay_A}=b/3+\sqrt{bc}/3\leq b$. The first inequality divided~by~$c$ is equivalent to $b/c\leq 4$, while the second divided by $b$ leads to the true inequality~$c/b\leq 4$. Thus~$w_A=b/3+c/3+2\sqrt{bc}/3$ precisely when~$b\leq 4c$.
When $b>4c$, by the same lemma, $w_A=\min \{c+b/2,b+c/2\}=c+b/2$, where~$c+b/2$ and~$b+c/2$ are the lengths of the parts of the perimeter containing $A$ respectively determined by the medians from $B$ and $C$.
This proves the expression for $w_A$. The ones for $w_B$ and $w_C$ are similar, except that the condition $a>4b$ for $w_C$ is impossible, since it implies $a>4c$ and by averaging we get~$a>2(b+c)$, contradicting the triangle inequality.

(ii) By part (i) and Lemma~\ref{L3}, the expression of $w=\min \{w_A,w_B,w_C \}$ depends on the position of $4c$ relative to $a$ and $b$. There are three possible cases to consider.

\emph{Case 1.} If $4c\geq a$, then $4c\geq b$, and so by part (i), $w_A=b/3+c/3+2\sqrt{bc}/3$ and~$w_B=a/3+c/3+2\sqrt{ac}/3$. The result in this case follows from the hypothesis that~$a\geq b\geq c$ being equivalent to $w_A\leq  w_B\leq w_C$.

\emph{Case 2.} If $a>4c\geq b$, then $w_A=b/3+c/3+2\sqrt{bc}/3$ and $w_B=c+a/2$. To prove that $w_A\leq w_B$, or $b/3+c/3+2\sqrt{bc}/3\leq c+a/2$, which is equivalent to~$\sqrt{4bc}\leq 2c-b+3a/2$, as $4c<a$, it is enough to prove that $\sqrt{ab}\leq 2c-b+3a/2$. By the means inequality, it suffices to show that $(a+b)/2\leq 2c-b+3a/2$, which leads to $3b\leq 4c+2a$. This is true, since $b\leq 4c$ and $2b\leq 2a$. The inequality $w_B\leq w_C$ is $c+a/2\leq a/3+b/3+2\sqrt{ab}/3$, which simplifies to $6c+a\leq 2b+4\sqrt{ab}$. For this, since $6c+a<3a/2+a=5a/2$ and $6b\leq 2b+4\sqrt{ab}$, it suffices to show that $5a/2\leq 6b$, or~$a\leq 12b/5$. Indeed, $a\leq b+c\leq b+b<12b/5$.

\emph{Case 3.} If $b>4c$, then also $a>4c$, hence $w_A=c+b/2$ and $w_B=c+a/2$, with the consequence that $w_A\leq w_B$. 
To prove the result, it suffices to show that $w_B\leq w_C$, which as in Case 2 is equivalent to $6c+a\leq 2b+4\sqrt{ab}$. For this, since $6c+a<3b/2+a$ and $6b\leq 2b+4\sqrt{ab}$, it suffices to show that $3b/2+a\leq 6b$, or~$a\leq 9b/2$. Indeed, $a\leq b+c\leq b+b/4<9b/2$.

(iii) Without loss of generality, we may assume that $a\geq b\geq c$. When $b\leq 4c$, by the means inequality, $w=b/3+c/3+2\sqrt{bc}/3\leq 2(b+c)/3$, with equality if and only if $b=c$. Furthermore, $2(b+c)/3\leq 4(a+b+c)/9$ is equivalent to~$b+c\leq 2a$, which is true, since $a\geq b,c$. Then the equality $w=4(a+b+c)/9$ happens precisely when $a=b=c$.
When $b>4c$, $w=c+b/2<4(a+b+c)/9$ is equivalent to $8a>b+10c$, which is true, since $a\geq b$ and $7a\geq 7b>28c>10c$.
\end{proof}

The proof of Theorem~\ref{T2}(i) shows that, when $a\geq b\geq c$, the W-line relative to $C$ is always a chord of the triangle, the W-line relative to $A$ is a chord if and only if $b\leq 4c$, and the W-line relative to $C$ is a chord if and only if $a\leq 4c$.

There are several differences between Theorem~\ref{T2} and the classical Winternitz theorem in a triangle. For example, part (ii) of the theorem shows that unlike its classical version, the fraction of the perimeter represented by $w$ is sides dependent.
Part (iii) of the same theorem shows that, unlike its classical version, where the inequality is an equality for all triangles, in the perimeter case, the equality is achieved only for the equilateral triangle.
And part (i) shows that unlike the classical Winternitz theorem, in the perimeter case, the numbers $w_A$, $w_B$, and $w_C$ are no longer equal for all triangles.

(The last statement requires more explanation. In the area case of the classical Winternitz theorem for a triangle $ABC$, $w_A$ is the minimum area of triangle $AMN$, for a variable line $MN$ through $G$, with $M$ on $AB$ and $N$ on $AC$, and similar definitions are given for $w_B$ and $w_C$. Let $\sigma $ denote the area of the triangle and let $w$ be the minimum area of one of the two parts cut from the triangle by a variable line through the centroid.
With this notation, the
classical Winternitz theorem for the concrete triangle~$ABC$ asserts that~$w=w_A=w_B=w_C=4\sigma /9$, and the three chords achieving~$w_A,w_B,w_C$ are precisely the parallels from $G$ to the sides.)

\medskip
We are now ready to proceed with the proof of Theorem~\ref{T1}, stated in the introduction.~Recall that the fraction of the perimeter represented by the shortest perimeter part cut by
a variable line through the centroid~$G$ of a triangle $ABC$ is a number~$m=w/(2s)$, where $0<m\leq 1/2$. The expression of $w$ in terms of the sides was computed in Theorem~\ref{T2}. Theorem~\ref{T1} asserts that the range of $m$ over all triangles is the interval $(3/10,4/9]$.

\begin{proof}[Proof of Theorem~\ref{T1}]
By Theorem~\ref{T2}(iii), the maximum of $m$ is 4/9 and this is achieved by the equilateral triangle and its  W-lines through $G$, which are the parallels from $G$ to the sides.

Focusing on the minimum, without loss of generality and up to similarity, we may assume that $a+b+c=1$ and $a\geq b\geq c>a-b$, where the last inequality is the triangle inequality. The substitution $a=1-b-c$ simplifies the quadruple inequality to\vspace{-.05in}
\begin{equation}\label{eq1}
(1-c)/2\geq b\geq c>1/2-b,
\end{equation}
and leaves the expression of $w=w_A$ in Theorem~\ref{T2} unchanged. We are then minimizing the function
\[
F(b,c)=w=\begin{cases}
b/3+c/3+2\sqrt{bc}/3,&\text{for }b\leq 4c,\\
c+b/2,&\text{for }b>4c,
\end{cases}
\]
over the closure of the region determined by the conditions \eqref{eq1} and shown in Figure~\ref{fig4}. The extremum values of $F$ over this region are obtained by standard multivariable calculus method, where the region is viewed as the union of its two sub-regions determined by the line $b=4c$ where the expression of $F$ changes. The observation that the partial derivative $F_b(b,c)$ is non-zero for all $(b,c)$ ensures no interior critical points for~$F$ in each subregion. One can easily prove that $F$ has no critical points on the six open segments bordering the two subregions, to deduce that the extremum values for~$F=w=m$ are corner point values. The evaluations of $F$ at the five corner points of the two subregions imply that $F$, hence $m$, has an unattained minimum~of~3/10 for~$(b,c)$ approaching $(4/10,1/10)$, that is, for $a:b:c$ approaching $5:4:1$; and we rediscover that~$m$ has a maximum of 4/9 at $(b,c)=(1/3,1/3)$, that is, for $a=b=c$.

By Theorem~\ref{T2}(i), when $b=4c$, the line through $G$ crossing $AB$ and $AC$ that achieves $w=w_A$ is the median from $B$, and this is a W-line relative to $A$.
\end{proof}
\begin{figure}[t]
\begin{center}
\begin{tikzpicture}[scale=12]\small
\draw [->] (-.04,.25) -- (.46,.25);\node [right] at (.46,.25) {$c,\quad (b=1/4)$};
\draw [->] (0,.21) -- (0,.54);\node [left] at (0,.53) {$b$};
\path [draw,fill=gray!30] (0,.5) -- (.25,.25) -- (.333,.333) -- (0,.5) -- cycle;
\draw [fill=black] (0,.5) circle [radius=.002in];\node [left] at (0,.5) {$(0,1/2)$};
\draw [fill=black] (.333,.333) circle [radius=.002in];\node [right] at (.333,.333) {$(1/3,1/3)$};
\draw [fill=black] (.25,.25) circle [radius=.002in];\node [below] at (.25,.25) {$(1/4,1/4)$};
\draw [fill=black] (.1,.4) circle [radius=.002in];\node [below left] at (.1,.4) {$(1/10,4/10)$};
\draw [fill=black] (.111,.444) circle [radius=.002in];\node [above right] at (.111,.444) {$(1/9,4/9)$};
\draw (.1,.4) -- (.111,.444);
\end{tikzpicture}\vspace{-.1in}
\end{center}
\caption{The region in the proof of Theorem~\ref{T1}.}\label{fig4}\vspace{-.1in}
\end{figure}
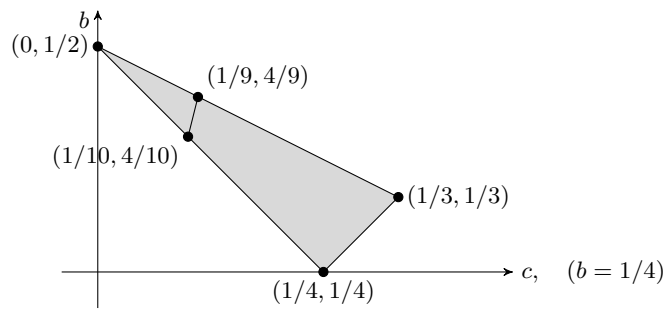

The proof of Theorem~\ref{T1} achieved the stated goal of the article. We close this study with a natural conjecture predicting that the lower bound of 3/10 in Theorem~\ref{T1} extends from triangles to planar convex sets, while clearly the upper bound of 4/9 does not.

\begin{PWC}
The fraction of the perimeter represented by the shortest perimeter part cut by a variable line through the centroid of a compact planar convex set is a number between 3/10 and 1/2. The bound 3/10 is approached by scales of triangles approaching the 5-4-1 degenerate triangle and their medians corresponding to the middle sides. The bound 1/2 is achieved by each line through the centroid of a centrally symmetric convex set.
\end{PWC}


\bibliographystyle{plain}

\begin{thebibliography}{12}

\bibitem{B} Blaschke, W. (1923) \textit{Vorlesungen \"uber Differentialgeometrie II. Affine Differentialgeometrie.} Berlin: Springer.

\bibitem{BCHM} Bose, P., Carmi, P., Hurtado, F., Morin, P. (2011). A generalized Winternitz theorem. \textit{J. Geom.} \textbf{100}: 29--35.

\bibitem{G} Gr\"umbaum, B. (1960). Partitions of mass distributions and of convex bodies by hyperplanes.
 \textit{Pacific J. Math.} \textbf{10}(4): 1257--1261.

\bibitem{N} Neumann, B. H. (1945). Partitions of mass distributions and of convex bodies by hyperplanes. \textit{J. Lond. Math. Soc.} \textbf{20}: 226--237.

\bibitem{SY} Shintar, A., Yaskin, V. (2021). A generalization of Winternitz's theorem and its discrete version. \textit{Proc. Amer. Math. Soc.} \textbf{149} (7): 3089--3104.

\bibitem{YB} Yaglom, I. M., Boltyanskii, V. G. (1960). \textit{Convex Figures}. Translated by Paul J. Kelly and Lewis F. Walton. New York, NY: Holt, Rinehart and Winston.
\end{thebibliography}


\end{document}